\documentclass[12pt,reqno]{amsart}
\usepackage{amsthm}
\usepackage{mathrsfs}
\usepackage{enumerate}
\usepackage{xspace}
\usepackage{tabulary}
\allowdisplaybreaks

\theoremstyle{definition}
\newtheorem{dfn}{Definition}

\theoremstyle{remark}
\newtheorem{rmk}{Remark}
\newtheorem{prob}{Problem}
\newtheorem{eg}{Example}

\theoremstyle{plain}
\newtheorem{thm}{Theorem}[section]
\newtheorem{prop}[thm]{Proposition}
\newtheorem{lem}[thm]{Lemma}
\newtheorem{cor}[thm]{Corollary}

\newcommand{\e}{\varepsilon}
\newcommand{\p}{\varphi}
\newcommand{\s}{\psi}

\newcommand{\vpr}{weave\xspace}
\newcommand{\vpg}{weaving\xspace}
\newcommand{\vpgc}{Weaving\xspace}
\newcommand{\vpgs}{weavings\xspace}
\newcommand{\vpgsc}{Weavings\xspace}
\newcommand{\vpa}{woven\xspace}
\newcommand{\vpac}{Woven\xspace}
\renewcommand{\epsilon}{\varepsilon}

\title{Weaving  Frames}

\author[Bemrose, Casazza, Gr\"ochenig, Lammers, and Lynch]{Travis Bemrose, Peter G. Casazza, Karlheinz Gr\"ochenig, Mark C. Lammers, Richard G. Lynch}

\thanks{The second and fifth authors were supported by NSF 1307685;
  NSF ATD 1042701; NSF ATD 00040683; AFOSR DGE51:
  FA9550-11-1-0245. The third author was
  supported by the  project P26273 - N25  of the
Austrian Science Fund (FWF)}

\begin{document}

\begin{abstract}
We study  an intriguing question in frame theory we call {\it Weaving Frames} that is partially motivated by
preprocessing of Gabor frames.  Two frames $\{\p_i\}_{i\in I}$ and $\{\psi_i
\}_{i\in I}$ for a Hilbert space ${\mathbb H}$ are {\it \vpa} if there are constants $0<A
\le B $ so that for every subset $\sigma \subset I$, the
family $\{\p_i\}_{i\in \sigma} \cup \{\psi_i\}_{i\in \sigma^c}$ is a frame for $\mathbb H$ with frame bounds
$A,B$. Fundamental properties of \vpa frames are developed and key
differences between \vpg Riesz bases and \vpg frames are considered.
In particular, it is shown that a Riesz basis  cannot be woven   with a redundant frame.  We also 
introduce an apparently weaker form of {\it weaving} but
show that it is equivalent to weaving. 
\vpgc frames has potential applications in wireless sensor networks that require distributed processing under different frames, as well as preprocessing of signals using Gabor frames. 
\end{abstract}

\maketitle



\section{Introduction}

This paper introduces a new problem  in frame theory called
\emph{\bfseries weaving frames}.  Two frames $\{\p_i\}_{i\in I}$ and
$\{\psi_i \}_{i\in I}$ for a Hilbert space ${\mathbb H}$ are (weakly)
{\it \vpa} if  for every subset $\sigma \subset I$, the 
family $\{\p_i\}_{i\in \sigma} \cup \{\psi_i\}_{i\in \sigma^c}$ is a frame for $\mathbb H$.

The concept is motivated by the following question in distributed
signal processing: given are two sets $\{\p_i\}_{i\in I} $ and
$\{\psi_i \}_{i\in I}$ of linear measurements with stable
recovery guarantees,  in mathematical terminology each  set is a
frame labeled by a node or sensor $i\in I$. At each sensor we 
measure a signal $x$ either with $\varphi _i $ or with $\psi _i$, so that the
collected information is the set of numbers $\{\langle x, \varphi
_i\rangle \} _{i \in \sigma } \cup \{ \langle x, \psi _i \rangle \}_{ i \in
\sigma ^c}$ for some subset $\sigma \subseteq I$. Can  $x$ still be
recovered robustly from these measurements, no matter 
which kind of measurement has been made at each node? In other words,
is the set $\{\p_i\}_{i\in \sigma} \cup \{\psi_i\}_{i\in \sigma^c}$ a
frame for all subsets $\sigma \subseteq I$? This question led us to
the definition of woven frames.

In this paper, we develop the fundamental properties of \vpa frames
for their own sake.  Naturally we hope that the notion of woven frames 
will be applicable to  wireless sensor networks which may be
subjected to distributed processing under different frames and
possibly in the preprocessing of signals using Gabor
frames. 

Let us briefly describe the main results about woven frames.  Let
$\Phi = \{\p_i\}_{i\in I}$ and
$\Psi = \{\psi_i \}_{i\in I}$ be two frames for a separable  Hilbert
space ${\mathbb H}$. 

(A) \emph{Uniform frame bounds.} If the weaving $\{\p_i\}_{i\in
  \sigma} \cup \{\psi_i\}_{i\in \sigma^c}$ is a
frame for every  subset $\sigma \subset I$,  then there exist
uniform frame bounds $A,B >0$ that work simultaneously for \emph{all}
weavings. In other words, the stability of the reconstruction does not
depend on  how the measurements are chosen
from the two frames $\Phi $ and $\Psi $. This surprising  fact shows
that the formal distinction of weakly woven frames and woven frames,
which we make in Sections~3 and~4, is not necessary. \\

(B) \emph{Woven Riesz bases.}  If $\Phi $ and $\Psi $ are two Riesz
bases such that every weaving $\{\p_i\}_{i\in
  \sigma} \cup \{\psi_i\}_{i\in \sigma^c}$ is a frame, then, in fact, $\{\p_i\}_{i\in
  \sigma} \cup \{\psi_i\}_{i\in \sigma^c}$ must already be a Riesz
basis for every $\sigma \subset I$. In Section~5 we also give a 
characterization of woven Riesz bases with a geometric flavor.  \\

(C) \emph{Existence of woven frames and perturbations.} The property that every weaving $\{\p_i\}_{i\in
  \sigma} \cup \{\psi_i\}_{i\in \sigma^c}$ is a
frame is rather  strong,  and it  seems that $\Phi $ and $\Psi $ must resemble each other in
some sense. We will show that if $\Psi $ is a perturbation of
$\Phi $, then $\Phi $ and $\Psi $ are indeed woven. For the technical
statements we will use several notions of perturbation, such as the
distance of the corresponding synthesis operators, or the almost
diagonalization of cross-Gramian matrix of the two frames, or the
perturbation by an invertible operator. As a consequence, every frame
with a reasonable condition number is woven with its canonical dual
frame. \\

 In the literature one finds other concepts which use multiple
 frames called \emph{quilted frames} in ~\cite{MD}.  These concepts require an underlying phase space and a notion
 of localization.  Quilted Gabor frames are then 
systems constructed from several globally defined frames by
restricting  these to certain sufficiently large  regions in the time-frequency or
time-scale plane.  Such frame constructions were 
investigated in detail by D\"orfler and Romero
in~\cite{MD,MD2,romero}. Except for the use of multiple frames,
quilted frames are seemingly unrelated to \vpg frames.





The paper is organized as follows:  Section~2 contains the basic
definitions about frames, and Section~3 introduces the new notion of
weaving frames. In Section~4 we give a characterization of weaving
frames that does not require universal frame bounds. In Section~5 we
consider the case of weaving Riesz bases and provide an abstract
characterization of when two Riesz bases are woven. In Sections~6
and~7 we provide sufficient conditions for weaving frames by means of
perturbation theory and diagonal dominance. Finally we speculate about
possible applications.

\section{Frame Theory Preliminaries}

A brief introduction to frame theory is given in this section, which
contains the necessary background for this paper. For a thorough
approach to the basics of frame theory, see \cite{petesbook,
  ole_book}. Throughout the paper, $\mathbb{H}$ will denote either a
finite or infinite dimensional Hilbert space while $\mathbb{H}^M$ will denote an  $M$-dimensional Hilbert space. Also, $I$ can represent a finite or countably infinite index set unless otherwise noted.

\begin{dfn}
A family of vectors $\Phi = \{\p_i\}_{i \in I}$ in a Hilbert space
$\mathbb{H}$ is said to be a \emph{\bfseries Riesz sequence} if  there are
constants $0 < A \leq B < \infty$ so that for all $\{c_i\}_{i \in I}
\in \ell^2(I)$, 
\[
A \sum_{i \in I} |c_i|^2 \leq \bigg\| \sum_{i \in I} c_i \p_i \bigg\|^2 \leq B \sum_{i \in I} |c_i|^2
\]
where $A$ and $B$ are the \emph{\bfseries lower Riesz bound} and
\emph{\bfseries upper Riesz bound}, respectively. If, in addition,
$\Phi $ is complete in $\mathbb{H}$, then it is a \emph{\bfseries
  Riesz basis} for $\mathbb{H}$.
\end{dfn}

An important, equivalent formulation for a Riesz
basis is that the vectors are the image of an orthonormal basis under
some invertible operator. That is, $\{\p_i\}_{i \in I}$ is a Riesz
basis for $\mathbb H$ if and only if there is an orthonormal basis
$\{e_i\}_{i \in I}$ for $\mathbb{H}$ and an invertible operator $T:
\mathbb{H} \to \mathbb{H}$ satisfying $T e_i = \p_i$ for all $i \in
I$. 

Riesz bases have proved to be  useful in those  applications in
which the assumption of orthonormality is too extreme, but the  uniqueness 
and  stability of the associated series expansion is still required.  

There are times when assuming the sequence is a Riesz basis is even too strong. In these cases, we work with
{\it frames} which are redundant family of vectors having proper subsets that span the space. Redundancy is the fundamental property
of frames which makes them so useful in practice since it
can be used to mitigate losses during transmission of a
signal, noise in the signal, and quantization errors, as well as being able to be adapted to particular signal characteristics.

\begin{dfn}
A family of vectors $\Phi = \{\p_i\}_{i \in I}$ in a Hilbert space $\mathbb{H}$ is said to be a \emph{\bfseries frame} if there are constants $0 < A \leq B < \infty$ so that for all $x \in \mathbb{H}$,
\[
A \|x \|^2 \leq \sum_{i\in I} |\langle x, \p_i\rangle|^2 \leq B \|x\|^2,
\]
where $A$ and $B$ are a choice \emph{\bfseries lower frame bound} and
\emph{\bfseries upper frame bound}, respectively. If only $B$ is
assumed to exist, then $\Phi $ is called a \emph{\bfseries Bessel sequence}. If
$A
= B = 1$, then $\Phi$  is a \emph{\bfseries Parseval frame}. 
The values $\{ \langle x,
\p_i\rangle \}_{i \in I}$ are called the \emph{\bfseries frame
  coefficients} of the vector $x\in \mathbb{H}$ with respect to the
frame $\Phi$.   
\end{dfn}

If $\Phi = \{\p_i\}_{i \in I}$ is a Bessel sequence  for $\mathbb{H}$, then the \emph{\bfseries analysis operator} of $\Phi$ is the operator $T \colon\mathbb{H} \to \ell^2(I)$ given by
\[
Tx = \{ \langle x, \p_i\rangle \}_{i\in I}
\]
and the associated \emph{\bfseries synthesis operator} is given by the adjoint operator $T^* \colon \ell^2(I) \to \mathbb{H}$ and satisfies
\[
T^{*} \{c_i\}_{i \in I}  = \sum_{i \in I} c_i \p_i.
\]
The \emph{\bfseries frame operator} $S\colon \mathbb{H}\to\mathbb{H}$ is the positive, self-adjoint, invertible operator defined by $S = T^* T$ and satisfies
\[
Sx = T^*Tx = \sum_{i \in I} \langle x, \p_i\rangle \p_i
\]
for every $x \in \mathbb{H}$. On the other hand, the \emph{\bfseries Gramian operator} $G : \ell^2(I)\to \ell^2(I)$ is the operator defined by $G = TT^{*}$ and has the matrix representation
\[
G = (\langle \p_i,\p_j\rangle )_{i,j\in I}.
\]
 These four operators are well-defined when the sequence $\Phi$ is assumed to be at least a $B$-Bessel sequence.

When $\Phi$ is a frame with bounds $A$ and $B$, the frame operator satisfies for every $x \in \mathbb{H}$,
\[ \langle Ax,x\rangle \le 
\langle Sx, x\rangle = \| T x\|^2 = \sum_{i \in I} |\langle x,\p_i\rangle |^2 \le \langle Bx,x\rangle,
\]
and hence operator inequality $A \cdot Id \leq S \leq B \cdot Id$ holds. Also, note that $\{S^{-1/2} \p_i\}_{i \in I}$ is a Parseval frame, called the \emph{\bfseries canonical Parseval frame of $\Phi$}. Finally, the norm of $S$ is $\|S\| = \|T^* T\| = \|T\|^2$.

\section{Getting Started}

In this section, the definition of \vpa frames is introduced. Then some results and examples are presented in regards to \vpg families of vectors. Throughout the rest of the paper for ease of notation, let
$$[m] = \{1,\ldots,m\} \quad \mbox{and} \quad [m]^c = \mathbb{N} \backslash [m] = \{m+1,m+2,\cdots\}$$
for a given natural number $m$. Also, denote by $[m,k] = [m+k]\backslash[m] = \{m+1,\ldots,m+k\}$ for every $m,k \in \mathbb{N}$. Here $\mathbb{N} = \{1,2,\ldots\}$.

\begin{dfn}
A family of frames $\{\p_{ij}\}_{i \in I}$ for $j \in [M]$ for a Hilbert space $\mathbb{H}$ is said to be \emph{\bfseries \vpa} if there are universal constants $A$ and $B$ so that for every partition $\{\sigma_j\}_{j \in [M]}$ of $I$, the family $\{\p_{ij}\}_{ i \in \sigma_j,j \in [M]}$ is a frame for $\mathbb{H}$ with lower and upper frame bounds $A$ and $B$, respectively. Each family $\{\p_{ij}\}_{ i \in \sigma_j,j \in [M]}$ is called a \emph{\bfseries \vpg}.
\end{dfn}

The first proposition of this section gives that every \vpg
automatically has a universal upper frame bound. 

\begin{prop}\label{upbound}
If each $\Phi _i = \{\p_{ij}\}_{ i \in I}$ is a Bessel  sequence for
$\mathbb{H}$ with  bounds $B_j$ for all $j \in [M]$, then every \vpg
is a Bessel sequence with  $\sum_{j=1}^M B_j$ as a Bessel bound. 
\end{prop}

\begin{proof}
For every partition $\{\sigma_j\}_{j \in [M]}$ of $I$ and every $x \in \mathbb{H}$, 
$$
\sum_{j=1}^M \sum_{i \in \sigma_j} \left| \langle x,
  \p_{ij}\rangle\right|^2 \leq \sum_{j=1}^M \sum_{i \in I} \left|
  \langle x, \p_{ij}\rangle\right|^2 \leq \sum_{j=1}^M B_j \cdot
\|x\|^2 \, ,
$$
yielding the desired bound.
\end{proof}

To verify that a finite number of frames $\Phi _i, i\in [M]$, is
woven, we therefore  only need to  
check that there is a universal lower bound for all \vpgs.

If all $\Phi _i$ are frames for $\mathbb{H}$, then the
bound in Proposition \ref{upbound} cannot be obtained, that is, 
the sum $\sum_{j \in [M]} B_j$ cannot be the smallest upper \vpg bound
if $B_j$ is optimal for its respective frame. Since the concept of weakly
\vpa is used to show this,  the proof is deferred until Section
\ref{weakq}, Remark \ref{notupbound}). However, the sum of the upper
bounds can be approached arbitrarily, as the next example
shows. Example \ref{arbbound} also shows that one cannot hope to
classify \vpa frames by placing restrictions on the \vpg bounds, even
if the frames considered are Parseval. 

\begin{eg}\label{arbbound}
There exist two Parseval frames that give \vpgs with arbitrary \vpg bounds. Let $\epsilon > 0$, set $\delta = (1 + \epsilon^2)^{-1/2}$, and let $\{e_1,e_2\}$ be the standard orthonormal basis of $\mathbb{R}^2$. Then the two sets
\[
\Phi = \{\p_i\}_{i=1}^4 =\{\delta e_1, \delta \epsilon e_1, \delta e_2, \delta \epsilon e_2\}
\]
and
\[ 
\Psi = \{\s_i\}_{i=1}^4 =\{\delta \epsilon e_1, \delta e_1, \delta \epsilon e_2, \delta e_2\}
\]
are Parseval frames, which are \vpa since any choice of $\sigma$ gives a spanning set (see Section \ref{weakq}, Theorem \ref{spanq}). Since they are Parseval, the universal upper frame bound for every \vpg can be chosen to be $2$ by Proposition \ref{upbound}. If $\sigma = \{2,4\}$ and $x \in \mathbb{R}^2$, then
\[
\sum_{i \in \sigma} |\langle x, \p_i\rangle|^2 + \sum_{i \in \sigma^c} |\langle x, \s_i\rangle|^2 = 2 \delta^2 \epsilon^2 \|x\|^2 = \dfrac{2 \epsilon^2}{1 + \epsilon^2} \|x\|^2
\]
which can be anywhere between $0$ and $2$ for arbitrary choice of $\epsilon \in (0,\infty)$.
\end{eg}

Clearly the property of woven frames is preserved under a bounded
invertible operator. This observation sometimes helps to simplify
proofs. 

\begin{cor}\label{parcor}
When considering whether two frames $\Phi=\{\p_i\}_{i\in I}$ and $\Psi=\{\s_i\}_{i\in I}$ are \vpa, it may always be assumed that one of them is Parseval by instead considering $\{S^{-1/2}\p_i\}_{i\in I}$ and $\{S^{-1/2} \psi_i\}_{i \in I}$ where $S$ is the frame operator of 
$\Phi$.
\end{cor}

The next example shows that in general, frames may be \vpa
without their canonical Parseval frames being \vpa. This also shows that applying two different operators to \vpa frames can give frames that are not \vpa.

\begin{eg}
There exist Riesz bases that are \vpa in which their canonical Parseval frames are not \vpa. Let $\Phi = \{e_1,e_2\}$ be the standard orthonormal basis for $\mathbb{R}^2$ and let $\Psi = \{\psi_1,\psi_2\}$ be defined by
$$
\psi_1= \left(\begin{array}{c} 1 \\ 1 \end{array} \right) \quad \mbox{and} \quad \psi_2 = \left(\begin{array}{c} 2 \\ 1 \end{array} \right).
$$
Note that $\Phi$ and $\Psi$ are \vpa. The frame operator of $\Psi$ is
\[
S = \psi_1\psi_1^* + \psi_2\psi_2^* = \left(\begin{array}{rr} 5 & 3 \\ 3 & 2 \end{array} \right).
\]
The canonical Parseval frame of $\Phi$ is itself. Note that because $\Phi$ is Parseval, any orthonormal basis forms a set of eigenvectors for its frame operator, the identity operator, and hence has the same eigenbasis as $S$. Finally, computing $S^{-1/2}$ gives
\[
S^{-1/2} = \left(\begin{array}{rr} 1 & -1 \\ -1 & 2 \end{array}\right)
\]
giving that the canonical Parseval frame of $\Psi$ is $\{S^{-1/2}\psi_1, S^{-1/2} \psi_2\} = \{e_2,e_1\}$ which is clearly not \vpa with $\Phi$.
\end{eg}

\section{Weakly \vpac Frames}\label{weakq}

In this section, it will be shown that if each \vpg is a frame, then necessarily 
there exists  a
universal lower frame bound for all weavings. Consequently it only
needs to  be
checked that each \vpg has a lower frame bound to show that a family
of frames are \vpa. 

\begin{dfn}
A family of frames $\{\p_{ij}\}_{i \in \mathbb{N}, j \in [M]}$ for a
Hilbert space $\mathbb{H}$ is said to be \emph{\bfseries weakly \vpa}
if for every partition  
$\{\sigma_j\}_{j \in [M]}$ of $\mathbb{N}$, the family $\{\p_{ij}\}_{i \in \sigma_j, j \in [M]}$ is a frame for $\mathbb{H}$.
\end{dfn}

The main question is whether or not weakly \vpa is equivalent to \vpa. The answer to this question in the finite dimensional case is obvious since there are only finitely many ways to partition the index set.

\begin{thm}\label{spanq}
Two frames $\{\p_i\}_{i = 1}^N$ and $\{\s_i\}_{i =1}^N$ for a
finite-dimensional Hilbert space $\mathbb{H}^M$ are \vpa if and only if for every $\sigma \subset [N]$, $\{\p_i\}_{i \in \sigma} \cup \{\s_i\}_{i \in \sigma^c}$ spans the space.
\end{thm}

Recall that  a finite  frame for $\mathbb{H}^M$ is said to be
\emph{\bfseries full spark},  if every $M$ element subset of the frame is linearly independent. 
We obtain the following   immediate condition for weaving finite frames.

\begin{cor}
If $\{\p_i\}_{i =1}^N$ is full spark in $\mathbb{H}^M$ and every subset of $\{\s_i\}_{i = 1}^N$ with $N-M$ elements spans, then these two frames are \vpa. In particular, if two full spark frames each have $N \geq 2M-1$ elements in an $M$ dimensional space, then they are necessarily \vpa.
\end{cor}

\begin{proof}
Let $\sigma \subset [N]$. If $|\sigma| \geq M$, then $\{\p_i \}_{i \in \sigma}$ spans and so $\{\p_i\}_{i \in \sigma} \cup \{\s_i\}_{i \in \sigma^c}$ is a frame. If $|\sigma| \leq M - 1$, then $|\sigma^c| \geq N - M + 1 \geq M$ and so $\{\s_i \}_{i \in \sigma^c}$ spans.
\end{proof}

\begin{rmk}
The corresponding statement  fails in infinite dimensions:  two frames
with the property that every finite subset is independent may not be
\vpa. To see this, take an orthonormal basis and a non-trivial permutation of it. 
\end{rmk}

The equivalence of woven and weakly woven frames  is significantly
more difficult to show for  frames in an infinite dimensional space.
As a preparation we need the following lemma.
 
\begin{lem}
Let $\{\p_i\}_{i = 1}^\infty$ and $\{\s_i\}_{i=1}^\infty$ be frames for a Hilbert space $\mathbb{H}$.  Assume that for every two disjoint finite sets $I,J \subset \mathbb{N}$ and every $\e >0$ there are subsets $\sigma,\delta \subset \mathbb{N} \backslash (I \cup J)$ with $\delta = \mathbb{N} \backslash (I \cup J \cup \sigma)$ so that the lower frame bound of $\{\p_i\}_{i\in I \cup \sigma}\cup\{\s_i\}_{i\in J \cup \delta}$ is less than $\e$.  Then there is a subset $\mathcal{A}\subset \mathbb{N}$ so
that $\{\p_i\}_{i\in \mathcal{A}}\cup\{\s_i\}_{i\in \mathcal{A}^c}$ is not a frame, i.e., these frames are not \vpa.
\end{lem}

\begin{proof}
Fix $\e >0$. By assumption, by letting $I_0 = J_0 = \emptyset$, we can choose $\sigma_1 \subset \mathbb{N}$ so that if $\delta_1 = \sigma_1^c$, a lower frame bound of
$\{\p_i\}_{i\in \sigma_1}\cup\{\s_i\}_{i\in \delta_1}$ is less than $\e$. Hence, there is a vector $h_1 \in \mathbb{H}$ with $\|h_1\|=1$ so that
\[ 
\sum_{i\in \sigma_1}|\langle h_1,\p_i\rangle|^2 + \sum_{i\in \delta_1}|\langle h_1,\s_i\rangle|^2 < \e.
\]
Since
\[ 
\sum_{i=1}^\infty |\langle h_1,\p_i\rangle|^2 + \sum_{i=1}^\infty|\langle h_1,\s_i\rangle|^2 < \infty,
\]
there is a $k_1 \in \mathbb{N}$ so that
\[
\sum_{i=k_1+1}^\infty |\langle h_1,\p_i\rangle|^2+\sum_{i=k_1+1}^\infty|\langle h_1,\s_i\rangle|^2 < \e.
\]
Let $I_1 = \sigma_1 \cap [k_1]$ and $J_1 = \delta_1 \cap [k_1]$. Then
$I_1 \cap J_1 = \emptyset$ and   $I_1 \cup J_1 = [k_1]$. By
assumption, there are subsets  $\sigma_2, \delta _2 \subset [k_1]^c$
with  $\delta_2
= [k_1]^c\setminus \sigma_2$ such that  a lower frame bound of
$\{\p_i\}_{i\in I_1\cup \sigma_2}\cup\{\s_i\}_{i\in J_1 \cup
  \delta_2}$ is less than $\e/2$. That is, there is a vector $h_2 \in
\mathbb{H}$ with $\|h_2\|=1$ so that 
\[ 
\sum_{i\in I_1\cup\sigma_2}|\langle h_2,\p_i\rangle|^2 + \sum_{i \in
  J_1 \cup \delta_2}|\langle h_2,\s_i\rangle|^2 < \dfrac{\e}{2} \, .
\]
 Similar to above, there is a $k_2 > k_1$ so that
\[
\sum_{i=k_2+1}^\infty |\langle h_2,\p_i\rangle|^2+\sum_{i=k_2+1}^\infty|\langle h_2,\s_i\rangle|^2 < \dfrac{\e}{2}.
\]
Set $I_2 = I_1 \cup \big(\sigma_2 \cap [k_2]\big)$ and $J_2 = J_1 \cup \big(\delta_2 \cap [k_2]\big)$. Note that $I_2 \cap J_2 = \emptyset$ and their union is  $I_2 \cup J_2 = [k_2]$. Continue inductively to obtain
\begin{itemize}
\item natural numbers $k_1 < k_2 < \cdots < k_n < \cdots$,
\item vectors $h_n \in \mathbb{H}$ with $\|h_n\| = 1$,
\item $\sigma_n \subset [k_{n-1}]^c$, $\delta_n = [k_{n-1}]^c \backslash \sigma_n$, and
\item $I_n = I_{n-1}\cup \big(\sigma_n \cap [k_n]\big)$, $J_n = J_{n-1} \cup \big(\delta_n \cap [k_n]\big)$
\end{itemize}
which satisfies both
\begin{align}\label{induc1}
\sum_{i\in I_{n-1}\cup\sigma_n}|\langle h_n,\p_i\rangle|^2 + \sum_{i \in J_{n-1} \cup \delta_n}|\langle h_n,\s_i\rangle|^2 < \dfrac{\e}{n}
\end{align}
and
\begin{align}\label{induc2}
\sum_{i=k_n+1}^\infty |\langle h_n,\p_i\rangle|^2+\sum_{i=k_n+1}^\infty|\langle h_n,\s_i\rangle|^2 < \dfrac{\e}{n}.
\end{align}
By construction $I_n \cap J_n = \emptyset$ and $I_n\cup J_n = [k_n]$ for all $n$ so that
\[ 
\bigcup_{i=1}^{\infty}I_i \sqcup \bigcup_{j=1}^{\infty}J_j = \mathbb{N},
\]
where ``$\sqcup$'' represents a disjoint union. Now set 
\[ 
\mathcal{A} = \bigcup_{i=1}^{\infty}I_i \, , \mbox{ and note } \mathcal{A}^c = \bigcup_{j=1}^{\infty}J_j.
\]
It follows by construction along with (\ref{induc1}) and (\ref{induc2}) that
\begin{align*}
\sum_{i \in \mathcal{A}} |\langle h_n, \p_i&\rangle|^2 + \sum_{i \in \mathcal{A}^c} |\langle h_n, \s_i\rangle|^2 \\
&= \bigg(\sum_{i \in I_n} |\langle h_n, \p_i\rangle|^2 + \sum_{j \in J_n} |\langle h_n, \s_i\rangle|^2\bigg) \\
& \quad +\bigg(\sum_{i \in \mathcal{A}\cap[k_n]^c} |\langle h_n, \p_i\rangle|^2 + \sum_{i \in \mathcal{A}^c\cap[k_n]^c} |\langle h_n, \s_i\rangle|^2\bigg)\\
& \leq \bigg(\sum_{i\in I_{n-1}\cup\sigma_n} |\langle h_n, \p_i\rangle|^2 + \sum_{i\in J_{n-1}\cup\delta_n} |\langle h_n, \s_i\rangle|^2\bigg)\\
&\quad + \bigg(\sum_{i=k_n+1}^\infty |\langle h_n,\p_i\rangle|^2+\sum_{i=k_n+1}^\infty|\langle h_n,\s_i\rangle|^2\bigg)\\
&< \dfrac{\e}{n} +  \dfrac{\e}{n} =  \dfrac{2\e}{n}
\end{align*}
so that a lower frame bound of $\{\p_i\}_{i\in A}\cup\{\s_i\}_{i\in A^c}$ is zero. Therefore, it is not a frame and the two original frames are not \vpa.
\end{proof}

\begin{cor}\label{cor101}
If the  frames $\{\p_i\}_{i = 1}^\infty$ and $\{\s_i\}_{i = 1}^\infty$
in  $\mathbb{H}$ are weakly woven, then there are 
disjoint finite sets $I,J \subset \mathbb{N}$ and  a constant $A>0$ so that for all
$\sigma,\delta \subset \mathbb{N} \setminus (I \cup J)$ and $\delta = \mathbb{N} \setminus (I \cup J \cup \sigma)$, the family $\{\p_i\}_{i\in I \cup \sigma} \cup \{\s_i\}_{i\in J \cup \delta}$ has lower frame bound $A$.
\end{cor}


Using Corollary \ref{cor101}, we can now prove  the main result of
this section. 

\begin{thm}\label{weakeq}
Given two frames $\{\p_i\}_{i=1}^{\infty}$ and
$\{\s_i\}_{i=1}^{\infty}$ for $\mathbb{H}$,  the following are equivalent:

\begin{enumerate}
\item[(i)]\label{weakeq1}  The two frames are \vpa.

\item[(ii)]\label{weakeq2}  The two frames are weakly \vpa.
\end{enumerate}
\end{thm}

\begin{proof}
Only (ii) $\Rightarrow$ (i) needs to be shown. By Corollary
\ref{cor101}, there are subsets $I,J \subset \mathbb{N}$ with $|I|,|J|
< \infty$ and $I\cap J = \emptyset$, and $A>0$ satisfying: 

\bigskip
($\dagger$) For every subset $\sigma \subset \mathbb{N} \setminus (I\cup J)$ with
$\delta =  \mathbb{N} \setminus (I\cup J\cup \sigma)$ the family $\{\p_i\}_{i\in I\cup \sigma}\cup\{\s_i\}_{i\in J\cup \delta}$ has lower frame bound $A$.  

\bigskip
To simplify the notation, we permute both frames so that $I=[q]$ and $J=[m]\backslash[q]$, or one of $I$ or $J$ is empty and the other is all of $[m]$. Note that permuting does not affect \vpg as long as it is done to both frames simultaneously.

\smallskip
\noindent {\bf Step 1}:   If
for each partition $I_{\alpha},J_{\alpha}$ of $[m]$ there is a constant $D_{\alpha}>0$ so that for every $\sigma
\subset [m]^c$ and $\delta = [m]^c\setminus \sigma$ the family $\{\p_i\}_{i\in I_{\alpha}\cup \sigma}\cup\{\s_i\}_{i\in J_{\alpha}\cup \delta}$ has lower frame bound $D_{\alpha}$, then the frames are \vpa with universal lower frame bound $A_0 = \min \{D_\alpha: \alpha\}$, which is positive 
since it is a minimum of a finite number of positive numbers.

Assume that the above does not hold. That is, there is a partition $I_1,J_1$ of $[m]$
such that for every $\e >0$ there are subsets $\sigma \subset [m]^c$ and $\delta=[m]^c\backslash \sigma$ so
that a lower frame bound of $\{\p_i\}_{i\in I_1\cup \sigma}\cup\{\s_i\}_{i\in J_1 \cup \delta}$ has lower frame bound less than
$\e$.  We will show that this yields a contradiction. 

\vspace{.3cm}
\noindent {\bf Step 2:} It is shown that for all $n \in \mathbb{N}$, there are
subsets $\sigma_n \subset [m]^c$ and $\delta_n = [m]^c
\setminus \sigma_n$ and unit vectors $h_n \in  \ell^2$ so that
\begin{align}\label{E101}
 \sum_{i\in I_1 \cup \sigma_n}|\langle h_n,\p_i\rangle|^2 + \sum_{i\in J_1 \cup
 \delta_n}|\langle h_n,\s_i\rangle|^2 < \frac{1}{n}.
\end{align}
Furthermore, the sets $\sigma_n$ and $\delta_n$ satisfy the following properties.

\smallskip
\begin{enumerate}
\item[(a)]\label{a1} For every $k = 1,2,\ldots$, either $m+k \in \sigma_n$ for all $n\geq k$ or $m+k \in \delta_n$ for all $n \geq k$.

\item[(b)]\label{a2} There is a $\sigma \subset [m]^c$ with $\delta = [m]^c \setminus \sigma$ so that $m+k\in \sigma$ implies that $m+k\in \sigma_n$ for all
$n\ge k$ or if $m+k\in \delta$ then $m+k\in \delta_n$ for all $n\ge k$.
\end{enumerate}

\smallskip
\noindent \textit{Proof of Step 2.} By assumption, for each $n$ there
exist subsets $\sigma_n , \delta _n  \subset [m]^c$ with  $\delta_n = 
[m]^c\setminus \sigma_n$ such that  the lower frame bound of
$\{\p_i\}_{i\in I_1 \cup \sigma_n}\cup\{\s_i\}_{i\in J_1 \cup
  \delta_n}$ is less than $1/n$.  Hence, there  are  vectors
$h_n$  of norm one so that 
\[  \sum_{i\in I_1 \cup \sigma_n}|\langle h_n,\p_i\rangle|^2 + \sum_{i\in J_1 \cup \delta_n}|\langle h_n,\s_i\rangle|^2 < \frac{1}{n},
\mbox{ for all } n\in \mathbb{N}.\]
Note that for each $n$, either $m+1 \in \sigma_n$ or $m+1 \in \delta_n$, so we can choose a subsequence 
$L_1 = \{l_{1j}\}_{j=1}^{\infty}$ of $\mathbb{N}$ such that one of the following must hold:
\begin{itemize}
\item  For every $k\in L_1$, $m+1 \in \sigma_k$.
\smallskip
\item  For every $k\in L_1$, $m+1 \in \delta_k$.
\end{itemize}
\noindent Similarly, there are subsequences $\{L_i\}_{i=1}^{\infty}$ of $\mathbb{N}$ with $L_i = \{l_{ij}\}_{j=1}^{\infty}$ satisfying:
\begin{itemize}
\item $L_{i+1}$ is a subsequence of $L_i$ for all $i=1,2,\ldots$.
\smallskip
\item For every $i = 1,2,\ldots$, either $m+i \in \sigma_k$ for all $k \in L_i$ with $k \geq l_{i,i}$ or $m + i \in \delta_k$ for all $k \in L_i$ with $k \geq l_{i,i}$.
\end{itemize}
Now set $L = \{l_{ii}\}_{i=1}^{\infty}$, so that $i \le l_{ii}$ for all $i$ and note that this is a subsequence of $\mathbb{N}$ satisfying:
$\{l_{ii}\}_{i=k}^{\infty}$ is a subsequence of $L_k$ for every $k\in\mathbb{N}$.

\smallskip
Now reindex the sigmas, $\sigma_{l_{nn}} \mapsto \sigma_{n}$ to obtain for each $n \in \mathbb{N}$, a subset $\sigma_n \subset [m]^c$ satisfying
\begin{eqnarray}\label{E102}
 \sum_{i\in I_1 \cup \sigma_n}|\langle h_n,\p_i\rangle|^2 + \sum_{i\in J_1 \cup \delta_n}|\langle h_n,\s_i\rangle|^2 < 
 \frac{1}{l_{nn}} \le\frac{1}{n}.
 \end{eqnarray}
Note that (\ref{a1}) is satisfied by construction of $\sigma_n$. To obtain (\ref{a2}), define $\sigma$ by putting $m+k\in \sigma$ if $m+k\in \sigma_n$ for all $n\ge k$ and $m+k\in \delta$ if $m+k\in \delta_n$ for all $n\ge k$. 
\vskip12pt

\noindent {\bf Step 3:}  It will be shown that by switching to a subsequence and reindexing, it can be assumed that
\[ h_n \rightarrow_{w} h \mbox{ and } h \not= 0.\]

\smallskip
\noindent \textit{Proof of Step 3.} Since the sequence $\{h_n\}_{n=1}^\infty$ is bounded, it has a weakly convergent subsequence $\{h_{n_i}\}_{i =1}^\infty$ that converges to some $h$. Reindex, $h_{n_i} \mapsto h_i$ and $\sigma_{n_i} \mapsto \sigma_i$ and notice that the properties proved in Step 2 still hold.

Fix $k \in \mathbb{N}$ so that $k > 2/A$, where $A$ is the constant
in ($\dagger$). Note that since $\{h_n\}_{n = 1}^{\infty}$ converges
weakly to $h$, it converges in norm to $h$ on finite dimensional
subspaces. In particular, there is an $N_k\in \mathbb{N}$ so that for
all $n \geq N_k > k$, 
\begin{align}\label{E103}
\sum_{i \in [m+k]}|\langle h-h_n,\p_i\rangle|^2 + \sum_{i \in [m+k]} |\langle h-h_n,\s_i\rangle|^2 < \frac{1}{2k}.
\end{align}
Now  ($\dagger$) implies  that, for every $n$,
\begin{align} \label{E104}
\sum_{i\in I\cup \sigma_n}|\langle h_n,\p_i\rangle|^2 + \sum_{i\in J\cup \delta_n}|\langle h_n,\s_i\rangle|^2 \ge A.
\end{align}
Therefore, if $n \geq N_k > k$, then by definition of $\sigma$ and $\delta$, $\sigma_{n}\cap[m,k] = \sigma\cap [m,k]$ and $\delta_{n}\cap[m,k] = \delta\cap [m,k]$, and thus inequalities (\ref{E102}), (\ref{E103}), and (\ref{E104}),   imply
\begin{align*}
\sum_{i\in I \cup \sigma}&|\langle h,\p_i\rangle|^2 + \sum_{i\in J \cup \delta}|\langle h,\s_i\rangle|^2 \\
&\geq \sum_{i\in I \cup (\sigma \cap [m,k])}|\langle h,\p_i\rangle|^2 + \sum_{i\in J \cup (\delta \cap [m,k])}|\langle h,\s_i\rangle|^2\\
&= \sum_{i \in  I \cup (\sigma_n \cap [m,k])}|\langle h,\p_i\rangle|^2 +\sum_{i\in J\cup (\delta_n \cap [m,k]} |\langle h,\s_i\rangle|^2 \\
&\geq \dfrac{1}{2} \bigg ( \sum_{i\in  I \cup (\sigma_n \cap [m,k])}|\langle h_n,\p_i\rangle|^2 + \sum_{i\in J\cup (\delta_n \cap [m,k])} |\langle h_n,\s_i\rangle|^2  \bigg  ) \\
&\quad -\bigg ( \sum_{i\in  I \cup (\sigma_n \cap [m,k])}|\langle h-h_n,\p_i\rangle|^2 + \sum_{i\in J\cup (\delta_n \cap [m,k])} |\langle h-h_n,\s_i\rangle|^2 \bigg )\\
&\geq \dfrac{1}{2} \bigg(\sum_{i\in I\cup \sigma_n}|\langle h_n,\p_i\rangle|^2 + \sum_{i\in J\cup \delta_n}|\langle h_n ,\s_i\rangle|^2 \bigg) \\
&\quad - \dfrac{1}{2}\bigg(\sum_{i\in \sigma_n \cap [m+k]^c}|\langle h_n,\p_i\rangle|^2  +\sum_{i\in \delta_n \cap [m+k]^c}|\langle h_n,\s_i\rangle|^2\bigg) - \frac{1}{2k}\\
&\geq \dfrac{1}{2} \cdot A - \dfrac{1}{2}\cdot\dfrac{1}{n} - \dfrac{1}{2k}\\
&\geq \dfrac{A}{2} - \dfrac{1}{k} > 0
\end{align*}
showing that $h \neq 0$.

\bigskip
\noindent {\bf Step 4:}  It is now shown that
\[ \sum_{i\in I_1 \cup \sigma}|\langle h,\p_i\rangle|^2 + \sum_{i\in J_1 \cup \delta}|\langle h,\s_i\rangle|^2 =0   .\]

\smallskip
\noindent \textit{Proof of Step 4.} 
By definition of $\sigma$ and $\delta$, $\sigma_{N_k}\cap[m,k] = \sigma\cap[m,k]$ and $\delta_{N_k}\cap[m,k] = \delta\cap[m,k]$, and therefore  inequalities $(\ref{E102})$ and $(\ref{E103})$ imply that
\begin{align*}
\sum_{i\in I_1 \cup \sigma}|&\langle h,\p_i\rangle|^2 + \sum_{i\in J_1 \cup \delta}|\langle h,\s_i\rangle|^2\\
&= \lim_{k\rightarrow \infty} \bigg( \sum_{i\in I_1 \cup (\sigma \cap [m,k])}|\langle h,\p_i\rangle|^2 + \sum_{i\in J_1 \cup (\delta \cap [m,k]\})}|\langle h,\s_i\rangle|^2 \bigg)\\
&= \lim_{k\rightarrow \infty} \bigg( \sum_{i\in I_1 \cup (\sigma_{N_k} \cap [m,k])}|\langle h,\p_i\rangle|^2 + \sum_{i\in J_1 \cup (\delta_{N_k} \cap [m,k]\})}|\langle h,\s_i\rangle|^2 \bigg)\\
&\leq 2\lim_{k\to \infty} \bigg( \sum_{i\in I_1 \cup \sigma_{N_k}}|\langle h_{N_k},\p_i\rangle|^2 + \sum_{i\in J_1 \cup \delta_{N_k}}|\langle h_{N_k},\s_i\rangle|^2 \bigg)\\
&\quad + 2\lim_{k\to \infty} \bigg( \sum_{i\in [m+k]}|\langle h - h_{N_k},\p_i\rangle|^2 + \sum_{i\in [m+k]}|\langle h - h_{N_k},\s_i\rangle|^2 \bigg)\\
&\leq 2 \lim_{k \to \infty} \dfrac{1}{N_k} + 2 \lim_{k\to\infty} \dfrac{1}{2k} = 0.
\end{align*}
Therefore, $h$ is not in the span of $\{\p_i\}_{i\in I_1 \cup \sigma}\cup\{\s_i\}_{i\in J_1\cup \delta}$ and hence this \vpg is not frame. It follows that the original frames are not weakly \vpa, so a contradiction is met, concluding the proof.
\end{proof}

\begin{rmk}\label{notupbound}
This section is concluded by showing that the upper bound in
Proposition \ref{upbound} cannot be obtained for woven frames. The case of two frames
is given, but the argument is easily extended to finitely
many. Suppose that $\Phi = \{\p_i\}_{i \in I}$ and $\Psi = \{\s_i\}_{i
  \in I}$ are frames for a Hilbert space $\mathbb{H}$ with optimal
upper frame bounds $B_1$ and $B_2$. Assume by way of
contradiction that $B_1 + B_2$ is the optimal upper \vpg bound. That
is, the smallest upper \vpg bound for all possible \vpgs. Then for a
fixed $\epsilon > 0$, one can choose a $\sigma \subset I$ and $\|x\| =
1$ so that  
\[
\sum_{i \in \sigma} |\langle x , \p_i \rangle|^2 + \sum_{i \in \sigma^c} |\langle x, \s_i \rangle |^2 \geq B_1 + B_2 - \epsilon.
\]
Since
\[ \sum_{i\in I}|\langle x,\p_i\rangle|^2 +
\sum_{i\in I}|\langle x,\psi_i\rangle|^2 \le B_1+B_2,\]
it follows that
\begin{align}\label{nubeq}
\sum_{i \in \sigma^c} |\langle x , \p_i \rangle|^2 + \sum_{i \in \sigma} |\langle x, \s_i \rangle |^2 \leq \epsilon.
\end{align}
Now  (\ref{nubeq}) implies that there are \vpgs for which the lower
frame bounds approach zero. Therefore, Theorem \ref{weakeq} gives that
$\Phi$ and $\Psi$ are not \vpa, which is  a contradiction.
\end{rmk}

\section{\vpgc Riesz Bases}

In this section we classify when Riesz bases and Riesz basic sequences
can be \vpa and provide a characterization in terms of distances between subspaces. 
\smallskip

We need a lemma in the case that $\sigma$ is finite  in order to prove Theorem \ref{riesz_seq}, which is in terms of general partitions.

\begin{lem}\label{lem20}
Suppose $\{\p_i\}_{i=1}^\infty$ and $\{\s_i\}_{i=1}^\infty$ are Riesz bases for $\mathbb H$ for which there are uniform constants $0 < A \leq B < \infty$ so that for every $\sigma \subset \mathbb{N}$, the family $\{\p_i\}_{i \in \sigma} \cup \{\s_i\}_{i \in \sigma^c}$ is a Riesz sequence with Riesz bounds $A$ and $B$. Then for every $\sigma \subset \mathbb{N}$ with $|\sigma|< \infty$,
the family $\{\p_i\}_{i\in \sigma} \cup \{\s_i\}_{i\in \sigma^c}$
is actually a Riesz basis,  that is, it spans $\mathbb H$.
\end{lem}

\begin{proof}
We proceed by induction on the cardinality of $\sigma $. The case
$|\sigma|=0$ being  obvious, we  assume
the result holds for every $\sigma $ with $|\sigma|=n$.

Now   let $\sigma \subset \mathbb{N}$ with
$|\sigma|=n+1$ and choose  $i_0 \in \sigma$.  Let $\sigma_1 = \sigma \setminus
\{i_0\}$, 
then $\{\p_i\}_{i\in \sigma_1} \cup \{\s_i\}_{i\in \sigma_1^c}$ is a Riesz
basis by the induction hypothesis.

We proceed by way of contradiction and assume that  $\{\p_i\}_{i\in \sigma}
\cup \{\s_i\}_{i\in \sigma^c}$ is not a Riesz basis. However, it is at least a 
Riesz sequence by assumption.  If
$$
\psi _{i_0} \in \mbox{span} (\{\p_i\}_{i\in \sigma} \cup \{\s_i\}_{i\in \sigma^c}),
$$
then
$$
\overline{\mbox{span}} \left ( \{\p_i\}_{i\in \sigma} \cup \{\s_i\}_{i\in \sigma^c} \right )
\supset
\overline{\mbox{span}}\left ( \{\p_i\}_{i\in \sigma_1} \cup \{\s_i\}_{i\in \sigma_1^c}\right
) = \mathbb{H},
$$
i.e., $\{\p_i\}_{i\in \sigma} \cup \{\s_i\}_{i\in \sigma^c}$ would be
a basis, which is assumed to not be the case. So it must be that 
$$
\psi _{i_0} \notin \mbox{span} \left ( \{\p_i\}_{i\in \sigma} \cup \{\s_i\}_{i\in
\sigma^c}
\right )
$$
from which it follows that
$$
\{\p_i\}_{i\in \sigma} \cup \{\s_i\}_{i\in \sigma^c} \cup \{\psi _{i_0}\}
$$
is a Riesz sequence in $\mathbb{H}$. Hence, because $\sigma_1^c = \sigma^c \cup \{i_0\}$,
$$
\{\p_i\}_{i\in \sigma_1} \cup \{\s_i\}_{i\in \sigma_1^c}
$$
cannot be  a Riesz basis, since we obtained it by deleting the
element 
 $\p _{i_0}$ from a 
Riesz sequence, yielding a contradiction.
\end{proof}

Next, the result from the previous lemma is extended to $\sigma$ of
arbitrary cardinality.

\begin{thm}\label{riesz_seq}
Suppose $\{\p_i\}_{i=1}^\infty$ and $\{\s_i\}_{i=1}^\infty$ are Riesz bases so that there are uniform constants $0 < A \leq B < \infty$ so that for every $\sigma \subset \mathbb{N}$, the family $\{\p_i\}_{i \in \sigma} \cup \{\s_i\}_{i \in \sigma^c}$ is a Riesz sequence with Riesz bounds $A$ and $B$. Then for every $\sigma \subset \mathbb{N}$ the family $\{\p_i\}_{i\in \sigma}
\cup \{\s_i\}_{i\in \sigma^c}$ is actually a Riesz basis.
\end{thm}

\begin{proof}
Assume, by way of contradiction, there is a $\sigma \subset \mathbb{N}$ with both $\sigma$ and $\sigma^c$ infinite, so that
$$
K = \overline{\mbox{span}} \left ( \{\p_i\}_{i\in \sigma}\cup \{\s_i\}_{i\in \sigma^c}
\right ) \not= \mathbb{H}.
$$
Choose a nonzero  $x\in K^{\perp}$.  Since $\{\s_i\}_{i\in I}$ is Bessel, by taking the tail of the series, there exists a $\sigma_1 \subset
\sigma$ with
$|\sigma_1|< \infty$ and
$$
\sum_{i \in \sigma\backslash \sigma_1} |\langle x,\s_i\rangle|^2 < \dfrac{A}{2} \|x\|^2.
$$
By Lemma \ref{lem20}, the family $\{\p_i\}_{i\in \sigma_1} \cup \{\s_i\}_{i\in \sigma\setminus \sigma_1}\cup \{\s_i\}_{i\in \sigma^c}$ is a Riesz basis with  Riesz basis bounds $A,B$
and therefore
\begin{align*}
A \|x\|^2 &\leq \sum_{i \in \sigma_1} |\langle x,\p_i\rangle|^2 + \sum_{i \in \sigma\backslash \sigma_1} |\langle x,\s_i\rangle|^2 + \sum_{i \in \sigma^c}|\langle x,\s_i\rangle|^2\\
&= \sum_{i \in \sigma\backslash \sigma_1} |\langle x,\s_i\rangle|^2\\
&\leq \dfrac{A}{2}\|x\|^2
\end{align*}
giving a contradiction.
\end{proof}

\begin{rmk}
If every \vpg of two Riesz bases is a frame sequence, it does not
follow that every \vpg is a Riesz basis, nor even a Riesz sequence. To
see this, let $\{\p_i\}_{i=1}^n$ be a Riesz basis and let $\pi$ be a
non-trivial permutation. Every \vpg $\{\p_i\}_{i\in
  I}\cup\{\p_{\pi(i)}\}_{i\in I}$ is a frame sequence but clearly does
not span the space.  
\end{rmk}

The next theorem says that if two  Riesz bases are woven, then every
weaving is in fact a Riesz basis, and not just a frame. 

\begin{thm}\label{frame_seq}
Suppose $\Phi = \{\p_i\}_{i=1}^\infty$ and $\Psi =
\{\s_i\}_{i=1}^\infty$ are Riesz bases and that  there is a uniform
constant $A > 0$ so that for every $\sigma \subset \mathbb{N}$, the
family $\{\p_i\}_{i\in\sigma} \cup \{\s_i\}_{i\in \sigma^c}$ is a
frame with lower frame bound $A$. Then for every $\sigma \subset
\mathbb{N}$, the family $\{\p_i\}_{i\in \sigma} \cup \{\s_i\}_{i\in
  \sigma^c}$ is actually a Riesz basis. 
\end{thm}

\begin{proof}
The proof is carried out in steps.

\bigskip
\noindent {\bf Step 1}:  For every $|\sigma|<\infty$, the family
$\{\p_i\}_{i\in \sigma}
\cup \{\s_i\}_{i\in \sigma^c}$ is a Riesz basis.

\smallskip
\noindent \textit{Proof of Step 1.}  We do the proof by induction on
$|\sigma|$ with
$|\sigma|=0$ clear.  
Now assume the result is true for all $|\sigma|=n$. Let $\sigma \subset I$ be so that $|\sigma| = n+1$ and let $i_0 \in \sigma$. Then $\{\p_i\}_{i \in \sigma \backslash \{i_0\}} \cup \{\s_i\}_{i \in \sigma^c \cup \{i_0\}}$ is a Riesz basis and therefore
$$
\{\p_i\}_{i \in \sigma\backslash\{i_0\}} \cup \{\s_i\}_{i \in \sigma^c}
$$ 
is a Riesz sequence spanning a subspace of codimension at least one. 

Now by assumption, $\{\p_i\}_{i \in \sigma}\cup\{\s_i\}_{i \in
  \sigma^c}$ is at least a frame. Since  the removal of the single
vector $\p_{i_0}$  yields a set that does not longer span
$\mathbb{H}$,   $\{\p_i\}_{i \in \sigma}\cup\{\s_i\}_{i \in
  \sigma^c}$  must actually be a Riesz basis~\cite{DS52}.  Furthermore,
its lower Riesz bound is $A$. This concludes the proof of Step 1. 

\bigskip

\noindent {\bf Step 2}:  Now consider the case that $|\sigma|= \infty$.

\smallskip

For this step choose $\sigma_1 \subset \sigma_2 \subset \cdots \subset
\sigma$ so that
$$
\sigma = \bigcup_{j=1}^{\infty}\sigma_j,
$$
and $|\sigma_j|< \infty$.  Now, for every $j=1,2,\ldots $ the family
$$
\{\p_i\}_{i\in \sigma_j}\cup \{\s_i\}_{i\in \sigma \backslash\sigma_j} \cup \{\s_i\}_{i\in
\sigma^c} = \{\p_i\}_{i\in \sigma_j}\cup \{\s_i\}_{i\in \sigma_j^c},
$$
is a Riesz basis with lower Riesz basis constant $A$.  If
$\{a_i\}_{i=1}^{\infty}
\in \ell^2$ and
$$
\sum_{i\in \sigma}a_i\p_i + \sum_{i\in \sigma^c}a_i\s_i = 0,
$$
then
\begin{align*}
0 &= \bigg\|\sum_{i\in \sigma}a_i\p_i + \sum_{i\in \sigma^c}a_i\s_i\bigg\|^2\\
&= \lim_{j\rightarrow \infty}\bigg\|\sum_{i\in \sigma_j}a_i\p_i + \sum_{i\in
\sigma_j^c} a_i\s_i\bigg\|^2\\
&\ge \lim_{j\rightarrow \infty} A \left ( \sum_{i\in \sigma_j}|a_i|^2
+ \sum_{i\in \sigma_j^c}|a_i|^2 \right )
\end{align*}
where the last inequality follows from the  Riesz basis property of  $\{\p_i\}_{i\in \sigma_j}\cup \{\s_i\}_{i\in \sigma_j^c}$.
So $a_i =0$ for every $i=1,2,\ldots $, implying that the synthesis operator for
the family
$\{\p_i\}_{i\in \sigma} \cup \{\s_i\}_{i\in \sigma^c}$ is bounded, linear, onto, 
and by the above it is also one-to-one.  Therefore, it is invertible and so the
family $\{\p_i\}_{i\in \sigma} \cup \{\s_i\}_{i\in \sigma^c}$ is a Riesz basis.
\end{proof}

The next results says that a frame (which is not
a Riesz basis) cannot be \vpa with a Riesz basis.

\begin{thm}\label{riesz_with_frame}
Let $\Phi = \{\p_i\}_{i=1}^\infty$ be a Riesz basis and let $\Psi = \{\s_i\}_{i=1}^\infty$ be a frame
for $\mathbb{H}$.  If $\Phi $ and $\Psi $  are \vpa ,  then $\Psi $  must actually
be a Riesz basis.
\end{thm}

\begin{proof}
Note that by Corollary \ref{parcor}, it may be assumed that $\{\p_i\}_{i=1}^\infty$ is  an orthonormal basis. By way of contradiction, assume that $\{\s_i\}_{i=1}^\infty$
is not a Riesz basis.
Without loss of generality it may be assumed that 
$\s_1 \in \overline{\mbox{span}} \{\s_i\}_{i=2}^{\infty}$.  Now, choose $n\in \mathbb{N}$ so that
$$
0 \leq d(\s_1,\ \mbox{span} \{\s_i\}_{i=2}^{n}) \le \e
$$
and let
$$
K_n = \left [ \,\mbox{span} \{\s_i\}_{i=2}^{n}\,\right ]^{\perp}.
$$
Then $K_n$ has codimension at most  $n-1$ in $\mathbb{H}$ and since $\{\p_i\}_{i=1}^\infty$
is an orthonormal basis,
$$
\dim\ \mbox{span} \{\p_i\}_{i=1}^{n} = n.
$$
So there exists $x\in \mbox{span} \{\p_i\}_{i=1}^{n} \cap K_n$ with $\|x\|=1$.
Now, if $\sigma^c = [n]$ then
$$
\sum_{i\in \sigma}|\langle x,\p_i\rangle |^2 = 0,
$$
while
$$
\sum_{i\in \sigma^c}|\langle x,\s_i \rangle |^2 = |\langle x,\s_1\rangle |^2
\le \e.
$$
So these two families are not \vpa.
\end{proof}

The next lemma and theorem provide some geometric intuition of what it
takes for Riesz bases to \vpr. First a notion of distance between
subspaces is introduced. 

\begin{dfn}
If $W_1$ and $W_2$ are subspaces of $\mathbb{H}$, let
\[
d_{W_1}(W_2) = \inf\{ \|x - y \| : x \in W_1, \, y \in S_{W_2}\}
\]
and
\[
d_{W_2}(W_1) = \inf\{ \|x - y \| : x \in S_{W_1}, \, y \in W_2\}
\]
where $S_{W_i} = S_\mathbb{H} \cap W_i$ and $S_\mathbb{H}$ is the unit sphere in $\mathbb{H}$. Then the \emph{\bfseries distance between $W_1$ and $W_2$} is defined as
\[
d(W_1,W_2) = \min\{d_{W_1}(W_2), d_{W_2}(W_1)\}.
\]
\end{dfn}

\begin{lem}\label{lemrb}
Suppose $\{\p_i\}_{i \in I}$ and $\{\s_j\}_{j \in J}$ are Riesz sequences in $\mathbb{H}$. The following are equivalent:
\begin{enumerate}
\item[(i)] $\{\p_i\}_{i \in I} \cup \{\s_j\}_{j \in J}$ is a Riesz sequence.
\item[(ii)] $D = d\big( \overline{\mbox{span}}\{\p_i\}_{i \in I}, \overline{\mbox{span}}\{\s_j\}_{j \in J}\big) > 0$.
\end{enumerate}
\end{lem}

\begin{proof}
To prove (i) $\Rightarrow$ (ii), let $A$ and $B$ be lower and upper Riesz bounds of $\{\p_i\}_{i \in I} \cup \{\s_j\}_{j \in J}$, respectively. Note that each original sequence $\{\p_i\}_{i \in I}$ and $\{\s_j\}_{j \in J}$ also has these bounds. If
$$
x = \sum_{i \in I} a_i \p_i \quad \mbox{and} \quad y = \sum_{j \in J} b_j \s_j
$$
with  $\|x\| = 1$ or $\|y\| = 1$, then
\begin{align*}
\|x - y\|^2 &= \bigg\| \sum_{i \in I} a_i \p_i - \sum_{j \in J} b_j \s_j \bigg\|^2\\
&\geq A \bigg(\sum_{i \in I} |a_i|^2 + \sum_{j \in J} |b_j|^2 \bigg)\\
&\geq A \left(\dfrac{1}{B} \bigg\| \sum_{i \in I} a_i \p_i \bigg\|^2 + \dfrac{1}{B} \bigg\| \sum_{j \in J} b_j \s_j \bigg\|^2 \right)\\
&\geq \dfrac{A}{B}
\end{align*}
so that $D^2 \geq A/B$, proving $(ii)$.

Now to prove (ii) $\Rightarrow$ (i), let $A$ and $B$ be universal lower and upper Riesz bounds, respectively, of the two original sequences. Let $\{a_i\}_{i \in I}$ and $\{b_j \}_{j \in J}$ be sequences of scalars so that
\[
\sum_{i \in I} |a_i|^2 + \sum_{j \in J} |b_j|^2 = 1.
\]
First assume  that $\sum_{i \in I} |a_i|^2 \geq 1/2$.  Then
\begin{align*}
\bigg\| \sum_{i \in I} a_i \p_i + \sum_{j \in J} b_j \s_j \bigg\|^2 &= \bigg\| \sum\limits_{i \in I} a_i \p_i \bigg\|^2 \left\| \dfrac{\sum\limits_{i \in I} a_i \p_i}{\bigg\| \sum\limits_{i \in I} a_i \p_i\bigg\|} + \dfrac{\sum\limits_{j \in J} b_j \s_j}{\bigg\| \sum\limits_{i \in I} a_i \p_i\bigg\|} \right\|^2\\
&\geq A \sum_{i \in I} |a_i|^2 D^2\\
&\geq \dfrac{AD^2}{2}.
\end{align*}
 If $\sum_{j \in J} |b_j|^2 \geq 1/2$, then a similar argument works,  and thus $\{\p_i\}_{i \in I} \cup \{\s_j\}_{j \in J}$ is a Riesz sequence with lower and upper bounds $AD^2/2$ and $B$, respectively.
\end{proof}

One more lemma is needed for the proof of Theorem \ref{thmrb5}. 

\begin{lem}\label{fsupprb}
Let $W_1$ and $W_2$ be subspaces of $\mathbb{H}$ and let $\{\p_i\}_{i \in I}$ and $\{\s_j\}_{j \in J}$ be a Riesz basis for $W_1$ and $W_2$, respectively. 
Then for every $\e > 0$, there is an $x \in S_{W_1}$ and $y \in W_2$ that are finitely supported on $\{\p_i\}_{i \in I}$ and $\{\s_j\}_{j \in J}$, respectively, so that
\[
\|x - y\| \leq d_{W_2}(W_1) + \e.
\]
\end{lem}

\begin{proof}
The proof is a  straightforward approximation argument. 
\end{proof}

The following theorem gives a geometric characterization of woven
Riesz bases. 

\begin{thm}\label{thmrb5}
If $\Phi=\{\p_i\}_{i=1}^\infty$ and $\Psi=\{\s_i\}_{i=1}^\infty$ are Riesz bases in $\mathbb{H}$, then the following are equivalent:
\begin{enumerate}
\item[(i)] $\Phi$ and $\Psi$ are \vpa.
\item[(ii)] For every $\sigma \subset \mathbb{N}$, $D_\sigma = d\big( \overline{\mbox{span}}\{\p_i\}_{i \in \sigma}, \overline{\mbox{span}}\{\s_i\}_{i \in \sigma^c}\big) > 0$.
\item[(iii)] There is a constant $C > 0$ so that for every $\sigma \subset \mathbb{N}$, $$D_\sigma = d\big( \overline{\mbox{span}}\{\p_i\}_{i \in \sigma}, \overline{\mbox{span}}\{\s_i\}_{i \in \sigma^c}\big) \geq C.$$
\end{enumerate}
\end{thm}

\begin{proof}
The implication  (i) $\Rightarrow$ (iii) follows from the proof of
Lemma \ref{lemrb} since if $A$ and $B$ are universal weaving bounds and $C=A/B$, then $D_\sigma \geq C$ for all $\sigma \subset \mathbb{N}$.

Conversely, for the implication (iii) $\Rightarrow$ (i), the proof of
Lemma \ref{lemrb} can be applied again, since if $A$ and $B$ are universal Riesz bounds of $\Phi$ and $\Psi$, then $AD_\sigma/2$ and $B$, are Riesz bounds for $\{\p_i\}_{i \in \sigma} \cup
\{\s_i\}_{i \in \sigma^c}$. Because $C \leq D_\sigma$, each weaving has Riesz bounds $AC/2$ and $B$, which are independent of $\sigma$. Furthermore, Theorem \ref{riesz_seq} gives that each \vpg is actually a Riesz basis,  and thus $\Phi$ and $\Psi$ are \vpa.


Since (iii) $\Rightarrow$ (ii) is obvious, all that remains is (ii)
$\Rightarrow$ (iii). Assume (ii) and by way of contradiction assume
that for every $n \in \mathbb{N}$, there are subsets $\sigma_n \subset
\mathbb{N}$ and elements $x_n, y_n \in \mathbb{H}$ of the form 
\[
x_n = \sum_{i \in \sigma_n} a_{in} \p_i \quad \mbox{and} \quad y_n = \sum_{i \in \sigma_n^c} b_{in} \s_i
\]
with either $\|x_n\|=1$ or $\|y_n\|=1$ that satisfy $\|x_n - y_n\|
\leq 1/n$. Without loss of generality, it may be assumed that
$\|x_n\|=1$ for all $n$, by passing  to a subsequence and  possibly switching the roles of $\Phi$ and
$\Psi$. Furthermore, by Lemma
\ref{fsupprb}, $x_n$ and $y_n$ may be assumed to be finitely
supported, and by switching to another subsequence and reindexing it
may be assumed that  
\[
x_n \rightarrow_{w} x = \sum_{i = 1}^\infty a_i \p_i \quad \mbox{and} \quad y_n \rightarrow_{w} y = \sum_{i = 1}^\infty b_i \s_i.
\]
The rest of the proof will be done in two steps.

\smallskip

\noindent {\bfseries Step 1:} $x = y = 0$.

\smallskip
\noindent \textit{Proof of Step 1.} For every $n \in \mathbb{N}$, either $1 \in \sigma_n$ or $1 \in \sigma_n^c$, so by switching to a subsequence, it may be assumed that $1 \in \sigma_n$ for all $n$ or that $1 \in \sigma_n^c$ for all $n$. Furthermore, switching to yet another subsequence, it may be assumed that $2 \in \sigma_n$ for all $n \geq 2$ or $2 \in \sigma_n^c$ for all $n \geq 2$. Continuing, a subsequence can be found (call it $\sigma_n$ again) satisfying:
\begin{itemize}
\item for every $i \leq k$, either $i \in \sigma_n$ for all $n \geq k$ or $i \in \sigma_n^c$ for all $n \geq k$.
\end{itemize}
Now define $\sigma \subset \mathbb{N}$ by
$$
\sigma = \{i \in \mathbb{N}: i \in \sigma_n \mbox{ for infinitely many }n \}.
$$
Since $x_n \rightarrow_{w}x$ and $\Phi$ is a Riesz bases, $a_{in} \to
a_i$ for any fixed $i$. Therefore, if $i \in \sigma$ is fixed, 
then $i \not \in \sigma _n^c$ and  $b_{in} = 0$ for all sufficiently
large $n$. Likewise,  if $i \in \sigma^c$, then
$a_{in} = 0$ for all large $n$. Hence, if $i \in \sigma$, then $b_i =
0$, and if $i \in \sigma^c$, then $b_i = 0$. 
Consequently, 
\begin{align}\label{xys}
x = \sum_{i \in \sigma} a_i \p_i \quad \mbox{and} \quad y = \sum_{i \in \sigma^c} b_i \s_i.
\end{align}
Next, since $x_n - y_n$ converges in norm to zero, it also does so weakly and thus 
\begin{align}\label{xy0}
\|x - y\| \leq \liminf_{n\to\infty} \|x_n - y_n\| = 0
\end{align}
implying that $x = y$. Finally, the assumption 
\[
d\big( \overline{\mbox{span}}\{\p_i\}_{i \in \sigma}, \overline{\mbox{span}}\{\s_i\}_{i \in \sigma^c}\big) > 0
\]
gives that $x = y = 0$ by (\ref{xys}) and (\ref{xy0}). Thus, $x_n \rightarrow_{w} 0$ and $y_n \rightarrow_{w} 0$.

\bigskip
\noindent {\bfseries Step 2:} $d\big( \overline{\mbox{span}}\{\p_i\}_{i \in \sigma}, \overline{\mbox{span}}\{\s_i\}_{i \in \sigma^c}\big) = 0$, giving a contradiction.

\smallskip
\noindent \textit{Proof of Step 2.}  First some notation is introduced. If $\eta$ and $\mu$ are finite subsets of $\mathbb{N}$, write $\eta \prec \mu$ if 
$$
\max\{i: i \in \eta\} \leq \min \{i : i \in \mu\}.
$$
Since $x_n$ and $y_n$ are finitely supported, there are finite subsets of $\mathbb{N}$, $\{\eta_k\}_{i = 1}^\infty$ subsets of $\sigma$, and $\{\mu_k\}_{k=1}^\infty$ subsets of $\sigma^c$ satisfying for every $k \in \mathbb{N}$:
\begin{enumerate}
\item[(a)] $\eta_k \cap \mu_k = \emptyset$.
\item[(b)] $\eta_k, \mu_k \prec \eta_{k+1}, \mu_{k+1}$.
\item[(c)] There is an $n_k$ so that 
$$
\bigg\| \sum_{i \in \bigcup_{j=1}^{k-1} \eta_j} a_{in_k} \p_i \bigg\| < \frac{1}{k} \quad \mbox{and} \quad \bigg\| \sum_{i \in \bigcup_{j=1}^{k-1} \mu_j} b_{in_k} \s_i \bigg\| < \frac{1}{k}.
$$
\end{enumerate}
With these properties in hand along with the fact that $x_{n_k}$ and $y_{n_k}$ are finitely supported, a standard epsilon thirds argument gives that
\[
d\big( \overline{\mbox{span}}\{\p_i\}_{i \in \eta_k}, \overline{\mbox{span}}\{\s_i\}_{i \in \mu_k}\big) \leq C_k
\]
where $C_k \to 0$ as $k \to \infty$, but
\[
d\big( \overline{\mbox{span}}\{\p_i\}_{i \in \sigma}, \overline{\mbox{span}}\{\s_i\}_{i \in \sigma^c}\big) \leq d\big( \overline{\mbox{span}}\{\p_i\}_{i \in \eta_k}, \overline{\mbox{span}}\{\s_i\}_{i \in \mu_k}\big)
\]
for all $k$, giving the desired contradiction.
\end{proof}

\begin{rmk}
In the literature one finds several notions for the distance or the
angle between two subspaces of a Hilbert space. For instance, the
angle between two subspaces $W_1, W_2 \subset \mathbb{H}$ is defined
as $$\alpha (W_1,W_2) =  \inf  \left\{\arccos\bigg(\dfrac{|\langle
    v,w\rangle|}{\|v\|\|w\|}\bigg) : v \in W_1, w \in W_2\right\},$$
see, e.g., ~\cite{deutsch}.  It is easy to see that $\alpha (W_1, W_2) 
> 0$ if and only if $d(W_1,W_2) >0$. Consequently,
Theorem~\ref{thmrb5} could also be formulated by means of the angle
between subspaces: two Riesz bases  $\Phi$ and $\Psi$
are \vpa , if and only if for every $\sigma \subset \mathbb{N}$, $
\alpha \big( \overline{\mbox{span}}\{\p_i\}_{i \in \sigma}, \overline{\mbox{span}}\{\s_i\}_{i \in \sigma^c}\big) > 0$.
\end{rmk}


\section{\vpgsc and Pertubations}

In this section it is shown frames that are small perturbations of
each other are \vpa. We state the results for two frames, but then give how they can be generalized to any finite number of them.

\begin{thm}\label{thm1}
Let $\Phi=\{\p_i\}_{i\in I}$ (respectively, $\Psi=\{\s_i\}_{i\in I}$) be frames for a
Hilbert space $\mathbb{H}$ with frame
bounds $A_1,B_1$ (respectively, $A_2,B_2$).  Assume there is a $0<\lambda <1$
so that
$$
\lambda (\sqrt{B_1}+\sqrt{B_2}) \le \frac{A_1}{2}
$$
and for all sequences of scalars $\{a_i\}_{i\in I}$ we have
\begin{equation}\label{eqn1}
\bigg\|\sum_{i\in I}a_i(\p_i -\s_i)\bigg\|\le \lambda \|\{a_i\}_{i\in I}\|.
\end{equation}
Then for every $\sigma \subset I$, the family $\{\p_i\}_{i\in \sigma^c}\cup
\{\s_i\}_{i\in \sigma}$
is a frame for $\mathbb{H}$ with frame bounds $\frac{A_1}{2},B_1+B_2$. That is, $\Phi$ and $\Psi$ are \vpa. 
\end{thm}

\begin{proof}
The proof will be carried out in steps, but first some notation is introduced. Let $T$
(respectively $R$) be the
synthesis operator for the frame $\{\p_i\}_{i\in I}$ (respectively,
$\{\s_i\}_{i\in I}$), and let $P_\sigma$ denote the orthogonal projection onto $\overline{\mbox{span}} \{e_i\}_{i \in \sigma}$, where $\{e_i\}_{i \in I}$ is the standard orthonormal basis for $\ell^2(I)$ and $\sigma \subset I$.
For each $\sigma \subset I$ let
$$
T_{\sigma}(\{a_i\}_{i\in I}) = TP_{\sigma}(\{a_i\}_{i\in I}) = \ \sum_{i\in \sigma}a_i\p_i,
$$
and
$$
R_{\sigma}(\{a_i\}_{i\in I}) = RP_{\sigma}(\{a_i\}_{i\in I}) = \sum_{i\in \sigma}a_i\s_i.
$$
In this notation, notice that the inequality in (\ref{eqn1}) becomes
$$
\|T-R\| \le \lambda.
$$
Observe that  $\|T_{\sigma}-R_{\sigma}\|\le
\|T-R\|$, and
$\|T_{\sigma}\| \le \|T\|$ and $\|R_{\sigma}\|\le \|R\|$ since $T_\sigma = TP_{\sigma}$ and $R_\sigma = RP_{\sigma}$.

\smallskip
\noindent {\bf Step 1}: For every $x\in \mathbb{H}$,
$$
\bigg\|\sum_{i\in \sigma}\langle x,\s_i\rangle \s_i - \sum_{i\in \sigma}\langle
x,\p_i\rangle \p_i\bigg\|=\|T_{\sigma}T_{\sigma}^{*}x - R_{\sigma}R_{\sigma}^{*}x\|
\le \frac{A_1}{2}\|x\|.
$$

\smallskip
\noindent \textit{Proof of Step 1.} Computing gives for every $x \in \mathbb{H}$ 
\begin{align*}
\|T_{\sigma}T_{\sigma}^{*}x - R_{\sigma}R_{\sigma}^{*}x\| 
&\le \|(T_{\sigma}T_{\sigma}^{*} - T_{\sigma}R_{\sigma}^{*})x\| +
\|(T_{\sigma}R_{\sigma}^{*}-R_{\sigma}R_{\sigma}^{*})x\|\\
&\le \|T_{\sigma}\|\|T_{\sigma}^{*}-R_{\sigma}^{*}\|\|x\|+
\|T_{\sigma}-R_{\sigma}\|\|R_{\sigma}^{*}\|\|x\|\\
&\le \|T\|\|T-R\|\|x\| + \|T-R\|\|R\|\|x\|\\
&\le (\|T\|\lambda + \|R\|\lambda)\|x\|\\
&\le \lambda (\|T\|+\|R\|)\|x\|\\
&\le \lambda (\sqrt{B_1} + \sqrt{B_2})\|x\|\\
&\le \frac{A_1}{2}\|x\|.
\end{align*}

\smallskip
\noindent {\bf Step 2}: The lower frame bound is $A_1/2$ for every \vpg.

\smallskip
\noindent \textit{Proof of Step 2.} For every $x\in \mathbb{H}$, it follows by applying Step 1 that
\begin{align*}
\bigg\|\sum_{i\in \sigma}\langle x,\s_i\rangle \s_i + &\sum_{i\in \sigma^c}\langle x,\p_i\rangle \p_i\bigg\| 
\\
&=\bigg\|\sum_{i\in I}\langle x,\p_i\rangle \p_i + \left(\sum_{i \in \sigma}\langle x,\s_i\rangle \s_i  - \sum_{i \in \sigma} \langle x,\p_i\rangle \p_i\right)\bigg\|\\
&\geq \bigg\|\sum_{i \in I}\langle x,\p_i\rangle \p_i\bigg\| - \bigg\|\sum_{i\in \sigma}\langle x,\s_i\rangle \s_i - \sum_{i \in \sigma} \langle x,\p_i\rangle \p_i\bigg\|\\
&\ge A_1\|x\| - \bigg\|\sum_{i\in \sigma}\langle x,\s_i\rangle \s_i - \sum_{i\in
\sigma}\langle
x,\p_i\rangle \p_i\bigg\|\\
&\ge A_1\|x\| - \frac{A_1}{2}\|x\| \\
&= \frac{A_1}{2}\|x\|.
\end{align*}

\smallskip
The upper frame bound of $ \{\s_i\}_{i\in \sigma} \cup
\{\p_i\}_{i\in\sigma^c}$ is at most $B_1+B_2$ by
Proposition~\ref{upbound}. Thus $\Phi $ and $\Psi $ are woven. 
\end{proof}

\begin{rmk}
Theorem \ref{thm1} can be generalized to a finite number of frames by
either requiring each
of the frames to be very close to one of them or by creating a ``chain''
where the first is close
to the second, the second to the third, and so on.
\end{rmk}

In the next proposition the perturbed frames is obtained as the image
of a bounded, invertible operator of a given frame.  

\begin{prop}\label{popw}
Let $\{\p_i\}_{i\in I}$ be a frame with bounds $A,B$ and let
$T$ be a bounded operator.  If $\|Id-T\|^2 < \frac{A}{B}$
then $\{\p_i\}_{i\in I}$ and $\{T\p_i\}_{i\in I}$ are \vpa.
\end{prop}

\begin{proof}
Note that $T$ is  invertible and thus $\{ T \varphi _i \} _{i\in I} $ is
automatically a frame.  For every $\sigma \subset I$ and for every $x\in \mathbb{H}$ we have by Minkowski's inequality and subadditivity of the square root function
\begin{align*}
\bigg ( \sum_{i\in \sigma}|\langle x,\p_i \rangle |^2
&+ \sum_{i\in \sigma^c}|\langle x,T\p_i\rangle |^2 \bigg )^{1/2}\\
&= \bigg ( \sum_{i\in \sigma}|\langle x,\p_i\rangle|^2 + \sum_{i\in
\sigma^c}|\langle x,\p_i\rangle - \langle (I-T^{*})x,\p_i \rangle |^2
\bigg )^{1/2}\\
&\ge \bigg ( \sum_{i\in I}|\langle x,\p_i\rangle |^2 \bigg )^{1/2}
- \left ( \sum_{i\in \sigma^c}|\langle (I-T^{*})x,\p_i\rangle |^2
\right )^{1/2}\\
&\ge \sqrt{A}\|x\| - \sqrt{B}\big\|(I-T^{*})x\big\|\\
&\ge \left ( \sqrt{A} - \sqrt{B} \|I-T^{*}\| \right ) \|x\|.
\end{align*}
Thus, $\{\p_i\}_{i\in \sigma} \cup \{T\p_i\}_{i\in \sigma^c}$
forms a frame having
\[
\Big(\sqrt{A} - \sqrt{B} \|I-T^{*}\|\Big) ^2 >0
\]
as its lower frame bound.
\end{proof}

\begin{rmk}
Proposition \ref{popw} can be extended to a finite number of invertible operators by making the assumption that the sum over all $j$ of $\|Id-T_j\|$ is less than $\sqrt{A/B}$. A nearly identical proof gives $\{\p_i\}_{i\in I}$, $\{T_j \p_i\}_{i \in I, j \in [n]}$ are \vpa.
\end{rmk}

By applying Proposition~\ref{popw} to a power of the frame operator, we obtain
many examples where a frame and its canonical dual frame are woven.

\begin{cor}
  Let $\Phi = \{ \varphi _i \} _{i\in I}$ be a frame with frame constants
  $A,B >0 $ and frame operator $S$.  If the condition number  $B/A$ is sufficiently close to
  one, e.g., $B/A<2$, then $\Phi $ and the scaled canonical dual frame $\Psi = \{ \frac{2AB}{A+B}S ^{-1} \varphi _i \}
  _{i\in I}$  are woven. Likewise  $\Phi $
  and  and the scaled  canonical Parseval frame $\tilde \Psi = \{
  \frac{2 \sqrt{AB}}{\sqrt{A} + \sqrt{B}}S ^{-1/2} \varphi _i \}
  _{i\in I}$ are woven.  
\end{cor}

\begin{proof}
  We apply Proposition~\ref{popw} to the operators $T =
  \frac{2AB}{A+B}S ^{-1}$ for the scaled dual frame and to $\tilde T =
  \frac{2 \sqrt{AB}}{\sqrt{A} + \sqrt{B}}S ^{-1/2}$ for the scaled
  Parseval frame. Since the spectrum of $S$ is contained in the
  interval $[A,B]$, the spectrum of $Id - T$ is contained in the
  interval $[\frac{A-B}{A+ B}, \frac{B-A}{A+ B}]$ and thus $\| Id -
  T\| \leq \frac{B-A}{B+A}$. This norm  is majorized by $(A/B)^{1/2}$,
  whenever $B/A \leq 2$. The proof for $\tilde T $ is similar and, in
  fact, yields the condition $B/A \leq (\sqrt{2} + 1)^2$. 
\end{proof}

\begin{rmk}
  Theorem~\ref{thm1} and Proposition~\ref{popw} should be compared
  with the corresponding statements for the perturbation of frames
  (e.g., in ~\cite{ole_book}). For instance, if $\Phi = \{ \varphi _i \}
  _{i\in I} $ is a frame with frame bounds $A,B >0$ and  $T$ is a bounded
  operator satisfying $\|Id - T \| <1$, then the set $\Psi = \{ T \varphi _i
  \}_{i\in I } $ is also a frame (because $T$ is invertible). Under
  the stronger condition $ \|Id - T \| <A/B$, the frame $\Phi $ and
  $\Psi $ are even woven. A similar remark applies to
  Theorem~\ref{thm1}. 
\end{rmk}
\section{Gramians}

In this section, the Gramian and its relation to \vpg frames is
considered. The first result says that two Riesz bases are \vpa as
long as the cross Gramian is almost  diagonal.

\begin{prop}\label{cross_gram}
Let $\{e_k\}_{k=1}^{\infty}$ be an orthonormal basis
for $\mathbb{H}$ and let $\{\p_{\ell}\}_{\ell =1}^{\infty}$ be
a Riesz basis for $\mathbb{H}$.  Let
$$
A = \left ( \langle \p_{\ell},e_k\rangle \right )_{k,\ell =1}^{\infty}
= D+R,
$$
be the cross Gramian where $D$ is the diagonal of $A$.  If the diagonal entries satisfy $\inf\p_{1\le i \le \infty}|D_{ii}| \ge \lambda$ and $\|R\|
\le \frac{\lambda}{2}$, then $\{e_k\}_{k=1}^{\infty},
\{\p_{\ell}\}_{\ell =1}^{\infty}$ are \vpa.
\end{prop}

\begin{proof}
Given $a \in \ell ^2$ and $\sigma \subset \mathbb{N}$ we have
\begin{align*}
\sum_{\ell \in \sigma^c}a_{\ell}\p_{\ell} &=
\sum_{\ell \in \sigma^c} \sum_{k=1}^{\infty}a_{\ell}
\langle \p_{\ell},e_k\rangle e_k\\
&= \sum_{k=1}^{\infty}\left ( \sum_{\ell \in \sigma^c}a_{\ell}
\langle \p_{\ell},e_k \rangle \right ) e_k\\
&= \sum_{k=1}^{\infty} \left ( A(I-P_{\sigma})a\right )_{k}e_k,
\end{align*}
where $P_{\sigma}$ is the diagonal projection onto span $\{e_k\}_{k\in
\sigma}$. Now we compute:

\begin{align*}
\bigg\|\sum_{k\in \sigma}a_ke_k + &\sum_{\ell \in \sigma^c}a_{\ell} \p_{\ell}\bigg\| = \| P_\sigma a + A(I-P_\sigma)a\|\\
&= \|P_\sigma a + D(I-P_\sigma)a + R(I-P_\sigma)a\|\\
&\geq \|P_\sigma a + D(I-P_\sigma)a\| - \|R (I-P_\sigma)a\|\\
&\geq \left(\|P_\sigma a\|^2 + \|D(I-P_\sigma) a\|^2 \right)^{1/2} - \|R\|\|(I-P_\sigma)a\|\\
&\geq \left(\|P_\sigma a\|^2 + \lambda^2\|(I-P_\sigma) a\|^2 \right)^{1/2} - \dfrac{\lambda}{2}\|(I-P_\sigma)a\|\\
&\geq \dfrac{1}{2}\left(\|P_\sigma a\|^2 + \lambda^2\|(I-P_\sigma) a\|^2 \right)^{1/2}\\
&\geq \dfrac{1}{2}\min(1,\lambda) \|a\| \, .
\end{align*}
This proves that $\{e_k\}_{k \in \sigma} \cup \{\p_\ell\}_{\ell \in
  \sigma^c}$ is a Riesz sequence with lower Riesz bound
$\min(1,\lambda^2)/4$ independent of  $\sigma \subseteq \mathbb{N}$.
Theorem \ref{riesz_seq} now implies  that the \vpgs are actually Riesz bases
with uniform bounds, and so the two sets are \vpa. 
\end{proof}

\begin{rmk}
Replacing $\{e_k\}_{k=1}^\infty$ with a Riesz basis is possible with
an appropriate change in required bounds. If it is replaced by a Riesz
basis $\{\s_k\}_{k=1}^\infty$, then there is a bounded invertible
operator so that $e_k = T^{-1}\s_k$. Therefore, the two sets
$\{e_k\}_{k=1}^\infty$ and $\{T^{-1}\p_\ell\}_{\ell=1}^\infty$ will be
\vpa if the correct bounds hold according to Proposition
\ref{cross_gram} and thus applying $T$ gives $\{\s_k\}_{k=1}^\infty$
and $\{\p_\ell\}_{\ell=1}^\infty$ are \vpa. 
\end{rmk}

\section{\vpgc Gabor Frames}

A potential application is the preprocessing of signals using Gabor frames. First, recall that a Gabor system for $L^2(\mathbb{R})$ is of the form
\[
\{ M_{bn} T_{am} g : m,n \in \mathbb{Z} \}
\]
where $a,b > 0$ are fixed parameters, $g \in L^2(\mathbb{R})$ is a
fixed window function, and the time-frequency shifts $M_{bn} T_{am}$
of $g$ are 
given by 
\[
M_{bn}T_{am}g(t) = e^{2\pi i bn t} g(t-am) \, .
\]
for  $a,b \in \mathbb{R}$, $m,n\in \mathbb{Z}$.  If such a system forms a frame, then it is called a Gabor frame. See \cite{gtfbook} for a thorough approach to time-frequency analysis and Gabor systems. The problem to consider is as follows.

\begin{prob}\label{gabor}
Given a fixed lattice generated by $a,b > 0$ with $ab < 1$ and rotated Gaussians $U_j g_{\alpha_j}$, where $g_{\alpha_i}(x) = e^{-\alpha_i x^2}$, are the Gabor frames $\{T_{am}M_{bn}U_jg_{\alpha_j}\}_{m,n\in \mathbb{Z}, j \in [M]}$ \vpa?
\end{prob}

We are unable to give a positive answer to Problem \ref{gabor}, however, we believe that any such family of Gabor frames is always \vpa.
\begin{rmk}
It seems that for two frames to be \vpa, in some sense they need to be close to one another. However, the problem is if $\mathscr{B}_1$ and $\mathscr{B}_2$ are arbitrary Bessel sequences, then $\Phi\cup\mathscr{B}_1$ and $\Psi\cup \mathscr{B}_2$ are still \vpa. That is, one would need to somehow reduce the frames to a ``minimal'' \vpa state for them to truly resemble each other.  In this paper we presented two sufficient conditions for weaving: {\it perturbation
theory}  and {\it diagonal dominance}.  It is possible that some kind of {\it localization} of the cross Gramian (without
diagonal dominance) would be sufficient.  This would also
answer Problem \ref{gabor}.  This also suggests that perhaps there is a converse to these results.  That is, if two frames
are {\it intrinsically} localized and \vpa, then is their
cross Gramian also localized (in the same matrix algebra?).
\end{rmk}
\noindent {\bf Acknowledgement}:  The authors would like
to thank Janet Tremain for many helpful discussions on
this research as well as providing several examples to us.


{\footnotesize \begin{center}
\begin{tabulary}{\linewidth}{C C C}
Travis Bemrose, & Peter G. Casazza, & Richard G. Lynch\\ 
tjb6f8@mail.missouri.edu, & pete@math.missouri.edu, & rglz82@mail.missouri.edu\\
\multicolumn{3}{c}{Department of Mathematics, University of Missouri-Columbia, USA}\end{tabulary}

\bigskip

\begin{tabulary}{\linewidth}{C C}
Karlheinz Gr\"ochenig, & Mark C. Lammers \\
karlheinz.groechenig@univie.ac.at, & lammersm@uncw.edu \\
Faculty  of Mathematics,& Department of Mathematics,\\
University of Vienna, Austria &University of North Carolina Wilmington
\end{tabulary}

\end{center}

\end{document}